\documentclass[10pt]{article}

\usepackage[a4paper,margin=1in]{geometry}
\usepackage{amssymb,amsmath,amsthm,epsfig}

\newtheorem{theorem}{Theorem}[section]
\newtheorem{lemma}[theorem]{Lemma}
\newtheorem{corollary}[theorem]{Corollary}
\newtheorem{definition}[theorem]{Definition}
\newtheorem{remark}[theorem]{Remark}
\newtheorem{note}[theorem]{Note}

\begin{document}

\begin{center}
{\Large \bf Euler-type Recurrence Relation for Arbitrary Arithmetical Function} \\[1.5em]

{\bf A. David Christopher} \\[0.5em]
{\small Department of Mathematics, The American College, Madurai, India} \\[0.3em]
{\tt davchrame@yahoo.co.in}
\end{center}

\vskip 20pt

\begin{abstract}
An interplay between the Lambert series and Euler's Pentagonal Number Theorem gives an Euler-type recurrence relation for any given arithmetical function. As consequences of this, we present Euler-type recurrence relations for some well-known arithmetic functions. Furthermore, we derive Euler-type recurrence relations for certain partition functions and sum-of-divisors functions using infinite product identities of Jacobi and Gauss.
\end{abstract}

\vskip 20pt

\section{Motivation and Statement of Results}


Euler's Pentagonal Number Theorem \cite{euler} states that, for complex $q$ with $|q|<1$,
\begin{equation}
\prod_{n=1}^{\infty}(1-q^n)=\sum_{m=0}^{\infty}\omega(m)q^m,
\end{equation}
where 
\[\omega(m)=\begin{cases} 1 &{\text {if }}  m=0;\\
(-1)^k &{\text{if }} m=\frac{3k^2\pm k}{2};\\
0 &{\text{otherwise.}}
\end{cases}
\]
This result, combined with the generating function for $p(n)$, the number of partitions of $n$,
\begin{equation}\label{gfp}
\prod_{m=1}^{\infty}\frac{1}{1-q^m}=\sum_{n=0}^{\infty}p(n)q^n
\end{equation}
gives a beautiful recurrence relation for $p(n)$:
\begin{equation}\label{pr}
\sum_{k=0}^{n}\omega(k)p(n-k)=0\text{   for each }n\geq 1.
\end{equation}
When Euler's Pentagonal Number Theorem is combined with the generating function for $q(n)$, the number of distinct partitions (partitions with distinct parts) of $n$,
\[\prod_{m=1}^{\infty}(1+q^m)=\sum_{n=0}^{\infty}q(n)q^n
\]
gives the following recurrence relation for $q(n)$, which is mentioned in \cite[p.~826]{abram}:
\begin{equation}\label{qetr}
\sum_{k=0}^n\omega(k)q(n-k)=\begin{cases}
\omega(\frac{n}{2})&\text{ if }n\text{ is even;}\\
0&\text{ otherwise.}
\end{cases}
\end{equation}
Using the operator $q\frac{d}{dq}\log$ in ({\ref{gfp}), one can get 
\[
\sum_{n=1}^{\infty}n\omega(n)q^n=\left(\sum_{m=0}^{\infty}\omega(m)q^m\right)\left(-\sum_{m=1}^{\infty}m\frac{q^m}{1-q^m}\right).
\]
Then, the Lambert series representation gives the following recurrence relation for $\sigma(n)$, the sum of positive divisors of $n$
\begin{equation}\label{sr}
\sum_{k=1}^{n}\sigma(k)\omega(n-k)=-n\omega(n).
\end{equation}
This recurrence relation involving pentagonal numbers was first realized by Euler \cite{euler2}. 

Any recurrence relation that assumes the pattern of the recurrence relations mentioned above is termed as Euler-type. For a derivation of the other Euler-type recurrence relations, see \cite{david}. In \cite{david}, Euler-type recurrence relation for the function $\tau_A(n)$ (resp. $\sigma_A(n)$), the number of (resp. the sum of) divisors of $n$ with divisors from a set of positive integers $A$  was derived by fusing Euler's Pentagonal Number Theorem and generating function of $\tau_A(n)$ (resp. $\sigma_A(n)$). The derivation technique wielded in \cite{david} is generalized here, and we obtain the first result of this paper. 
\begin{theorem}\label{etr}
Let $g$ be an arithmetical function. Then we have
\begin{equation}\label{etr-formula}
\sum_{k=1}^{n}g(k)\omega(n-k)=\sum_{m+k=n; m\geq 1;k\geq 0}f(m)\left(\omega(k)+\omega(k-m)+\omega(k-2m)+\cdots \right),
\end{equation}
where 
\[g(n)=\sum_{d|n}f(d).
\] 
\end{theorem}
\begin{remark}
\emph{The highlighting feature of the above recurrence relation is that: though $g$ is a divisor sum of $f$, if $f$ can be put as an expression without factorization being involved, then one can obtain a recurrence relation for $g$ without factoring each $k$. On the other hand, if $g$ is a known simple expression where factorization is not needed and, though $f$ may depend on factorization, then using (\ref{etr-formula}) one can obtain a recurrence relation for $f$ without factoring each $m$}.
\end{remark}
In Section 2, we present a proof for Theorem \ref{etr}, and we employ Theorem \ref{etr}  to obtain  recurrence relations  of Euler-type for several arithmetical functions including the Euler's phi function $\phi(n)$, the number-of-divisors function $\tau(n)$, the Liouville's function $\lambda(n)$ and the M\"{o}bius function $\mu(n)$.

In Section 3, manipulation of infinite product identities plays a vital role.  The following result is established in Section 3 by making use of the Jacobi's triple product identity.

\begin{theorem}\label{etnew1}
Let $q(n)$ and $qq(n)$ be the number of distinct partitions of $n$  and the number of distinct partitions of $n$ with odd parts, respectively. We have
\begin{enumerate}
\item[(a)] 
\begin{equation}
\sum_{k=0}^{n}(-1)^{k}qq(k)\omega(n-k)=\begin{cases}
1 &\text{if }n=0\\
2 &\text{if }\delta_s(n)=1\text{ and }n\equiv 0(mod\ 2)\\
-2 &\text{if }\delta_s(n)=1\text{ and }n\equiv 1(mod\ 2)\\ 
0 &\text{otherwise,}
\end{cases}
\end{equation}
\item[(b)]
\begin{equation}
q(n)+2\sum_{k\geq 1}(-1)^k\delta_s(k)q(n-k)=\omega(n),
\end{equation}
where
\[\delta_s(n)=\begin{cases}
1 &\text{if $n$ is a square number}\\
0 &\text{otherwise.}
\end{cases}
\]
\end{enumerate}
\end{theorem}
Furthermore, in Section 3, the following identities due to Gauss \cite{gauss} were used to establish Theorem \ref{etnew2} and Theorem \ref{comp}.
\begin{enumerate}
\item
\begin{equation}\label{gauss1}
\sum_{n=0}^{\infty}q^{\frac{n(n+1)}{2}}=\frac{\prod_{m=1}^{\infty}(1-q^{2m})}{\prod_{m=1}^{\infty}(1-q^{2m-1})},
\end{equation}
\item
\begin{equation}\label{gauss2}
\sum_{n=-\infty}^{\infty}(-1)^nq^{n^2}=\prod_{m=1}^{\infty}\frac{1-q^m}{1+q^m}.
\end{equation}
\end{enumerate}
\begin{theorem}\label{etnew2}
We have
\begin{enumerate}
\item[(a)]
\begin{equation}
\sum_{k=0}^{\lfloor\frac{n}{2}\rfloor}q(n-2k)\omega(k)=\begin{cases}
1&\text{ if }\delta_t(n)=1\\
0&\text{ otherwise,}
\end{cases}
\end{equation}
\item[(b)]
\begin{equation}
\sum_{k=0}^{\lfloor\frac{n}{2}\rfloor}p(k)\delta_t(n-2k)=q(n),
\end{equation}
\item[(c)]
\begin{equation}
\sum_{k=0}^{n}(-1)^kqq(k)\delta_t(n-k)=\begin{cases}
\omega(\frac{n}{2})&\text{ if }n\text{ is even}\\
0&\text{ otherwise},
\end{cases}
\end{equation}
where 
\[\delta_t(n)=\begin{cases}1&\text{ if  }n=\frac{m(m+1)}{2}\\
0&\text{ otherwise.}
\end{cases}
\]
\end{enumerate}
\end{theorem}
Recall that, a  {\it composition} of a positive integer $n$ is an ordered sequence of  positive integers, say $\pi=(a_1,a_2,\cdots ,a_k)$, such that $a_1+a_2+\ldots+a_k=n$. Composition $\pi$ is said to be {\it distinct} if $a_i\neq a_j$ for every $i\neq j$. Composition $\pi$ is said to be $relatively\  prime$ if $\gcd(a_1,a_2,\cdots ,a_k)=1$.
\begin{theorem}\label{comp}
Let $s(n)$ and $t(n)$ be the number of compositions of $n$ with square numbers as parts and number of compositions of $n$ with triangular numbers as parts, respectively. We have
\begin{enumerate}
\item[(a)]
\begin{equation}
\sum_{k=0}^{n}(-1)^ks(k)\left(3q(n-k)-\omega(n-k)\right)=2q(n),
\end{equation} 
\item[(b)]
\begin{equation}
\sum_{k=0}^{n}t(k)\left(2(-1)^{n-k}qq(n-k)-\omega'(n)\right)=(-1)^nqq(n),
\end{equation}
where 
\[w'(n)=\begin{cases}
\omega(\frac{n}{2})&\text{ if }$n$\text{ is even}\\
0&\text{ otherwise.}
\end{cases}
\]
\end{enumerate}
\end{theorem}
In Section 4, we wield the logarithmic derivative technique, in the way of  Euler, to get the following result.
\begin{theorem}\label{ld1}
Let $\sigma_o(n)$(resp. $\sigma_e(n)$) be the sum of odd (resp. even) positive divisors of $n$. We have
\begin{enumerate}
\item[(a)]
\begin{equation}
\sigma(n)+\sigma_o(n)+2\sum_{k\geq 1}(-1)^k\delta_s(k)(\sigma(n-k)+\sigma_o(n-k))=2(-1)^{n+1}n\delta_s(n),
\end{equation} 
\item[(b)]
\begin{equation}
\sum_{k=1}^n\left(\sigma_o(k)-\sigma_e(k)\right)\delta_t(n-k)=n\delta_t(n),
\end{equation}
\item[(c)]
for $n\geq 2$, we have
\begin{equation}
\sigma(n)+\sigma_s(n)+2\sum_{k\geq 1}(-1)^k\delta_s(k)(\sigma(n-k)+\sigma_s(n-k))=2(-1)^{n+1}n\delta_s(n),
\end{equation}
where 
\[\sigma_s(n)=\sum_{d|n}(-1)^{d-1}\frac{n}{d}.
\]
\end{enumerate}
\end{theorem}
The meaning of the notation specified in this section stands throughout the article. 
\section{ Proof and applications of Theorem \ref{etr}}
The following lemma forms a crucial part of the proof. 
\begin{lemma}\label{l1}
We have 
\begin{equation}
\frac{\prod_{n=1}^{\infty}(1-q^n)}{1-q^k}=\sum_{m=0}^{\infty}\left(\omega(m)+\omega(m-k)+\omega(m-2k)+\cdots \right)q^m
\end{equation}
with $\omega(m)=0$ for every $m<0$.
\end{lemma}
\begin{proof}
Let 
\[\frac{\prod_{n=1}^{\infty}(1-q^n)}{1-q^k}=\sum_{m=0}^{\infty}\omega_k(m)q^m.
\]
Then, by Euler's Pentagonal Number Theorem, we  have
\[\omega_k(n)-\omega_k(n-k)=\omega(n).
\]
Repeated application of the above relation gives
\[\omega_k(n)=\omega(n)+\omega(n-k)+\omega(n-2k)+\cdots
\]
as expected. 
\end{proof}
\subsection{ Main part of the proof}
From the Lambert series representation, we have
\[
\sum_{n=1}^{\infty}g(n)q^n=\sum_{m=1}^{\infty}f(m)\frac{q^m}{1-q^m},
\]
where $g(n)=\sum_{d|n}f(d)$. 

Then multiplying both sides by the term $\prod_{n=1}^{\infty}(1-q^n)$, and using Euler's Pentagonal Number Theorem with Lemma \ref{l1} gives the following equalities:
\begin{align*}
\sum_{n=1}^{\infty}\left(\sum_{k=1}^{n}g(k)\omega(n-k)\right)q^n&=\sum_{m=1}^{\infty}f(m)\frac{\prod_{n=1}^{\infty}(1-q^n)}{1-q^m}q^m\\ 
&=\sum_{m=1}^{\infty}f(m)q^m\left(\sum_{k=0}^{\infty}\omega(k)+\omega(k-m)+\omega(k-2m)+\cdots \right)q^k.
\end{align*}
Equating the coefficients of $q^n$ gives the desired end. 
\subsection{ Recurrence relation for $\phi$, $\tau$, $\lambda$ and $\mu$}
As the first application of Theorem \ref{etr}, we present a recurrence relation for the famous Euler's phi function $\phi(n)$.
\begin{theorem}
\begin{equation}
\sum_{k=1}^{n}k\omega(n-k)=\sum_{m+k=n; m\geq 1;k\geq 0}\phi(m)\left(\omega(k)+\omega(k-m)+\omega(k-2m)+\cdots \right).
\end{equation}
\end{theorem}
\begin{proof}
Applying Gauss identity
\[\sum_{d|n}\phi(d)=n
\]
in Theorem \ref{etr}, we get the above relation. 
\end{proof}
Next is a recurrence relation for the number-of-divisors function $\tau(n)$ as an immediate consequence of Theorem \ref{etr}.
\begin{theorem}\label{tau}
We have
\begin{equation}
\sum_{k=1}^{n}\tau(k)\omega(n-k)=\sum_{m+k=n; m\geq 1;k\geq 0}(\omega(k)+\omega(k-m)+\cdots ).
\end{equation}
\end{theorem}
The Liouville's function, denoted by $\lambda(n)$, is defined as 
\begin{equation}
\lambda(n)=\begin{cases} 1 &\text{if } n=1;\\
(-1)^k &\text{if the number of prime factors of $n$ counted with multiplicity is $k$}.
\end{cases}
\end{equation}
Now we present a recurrence relation for $\lambda(n)$, which is based on the following well-known relation:
\begin{equation}
\sum_{d|n}\lambda(d)=\delta_s(n).
\end{equation}
\begin{theorem}
We have
\begin{equation}
\sum_{k=1}^{n}\delta_s(k)\omega(n-k)=\sum_{m+k=n; m\geq 1;k\geq 0}\lambda(m)(\omega(k)+\omega(k-m)+\cdots ).
\end{equation}
\end{theorem}
From the following well-known relation:
\begin{equation}
\sum_{d|n}\mu(d)=\begin{cases}
1&\text{ if }n=1\\
0&\text{ otherwise},
\end{cases}
\end{equation}
we have the following recurrence relation for the M\"{o}bius function $\mu(n)$. 
\begin{theorem}
For $n>1$, we have
\begin{equation}
\omega(n-1)=\sum_{m+k=n; m\geq 1;k\geq 0}\mu(m)(\omega(k)+\omega(k-m)+\cdots ).
\end{equation}
\end{theorem}
\subsection{ Relatively prime (distinct) partitions and compositions}
\begin{theorem}Let $p_\psi(n)$ be the number of relatively prime partitions of $n$. We have
\begin{equation}
-\omega(n)=\sum_{m+k=n; m\geq 1;k\geq 0}p_\psi(m)(\omega(k)+\omega(k-m)+\cdots ).
\end{equation}
\begin{proof}
Mohamed El Bachraoui \cite{moh} noted that 
\[p(n)=\sum_{d|n}p_\psi(d).
\]
Keeping this observation, the result follows as a consequence of Theorem \ref{etr} and recurrence relation \ref{pr}.
\end{proof}
\end{theorem}
\begin{theorem}\label{qrr}
Let $q_\psi(n)$ be the number of relatively prime distinct partitions of $n$. We have
\begin{equation}
\sum_{m+k=n; m\geq 1;k\geq 0}q_\psi(m)(\omega(k)+\omega(k-m)+\cdots )=\begin{cases}-\omega(n)&\text{ if $n$ is odd;}\\
-\omega(n)+\omega(\frac{n}{2})&\text{ if $n$ is even.}
\end{cases}
\end{equation}
\end{theorem}
\begin{proof}
Using the recurrence relation (\ref{qetr}) in Theorem \ref{etr}, together with the following observation:
\[q(n)=\sum_{d|n}q_\psi(d),
\]
completes the proof. 
\end{proof}
\begin{theorem}
Let $c_\psi(n)$ be the number of relatively prime compositions of $n$. We have
\begin{equation}
\sum_{k=1}^{n}2^{k-1}\omega(n-k)=\sum_{m+k=n; m\geq 1;k\geq 0}c_\psi(m)(\omega(k)+\omega(k-m)+\cdots ).
\end{equation}
\end{theorem}
\begin{proof}
Let $c(n)$ be the number of compositions of $n$. Gould\cite{gould} noted that
\begin{equation}
c(n)=\sum_{d|n}c_\psi(d)
\end{equation}
and
\begin{equation}\label{comp-1}
c(n)=2^{n-1}.
\end{equation}
Now the result follows while presenting the above relations in Theorem \ref{etr}.
\end{proof}
\begin{theorem}\label{restrictedcpsi}
Let $c_\psi(n,r)$ be the number of relatively prime compositions of $n$ with exactly $r$ parts. We have
\begin{equation}
\sum_{k=1}^{n}{{k-1}\choose{r-1}}\omega(n-k)=\sum_{m+k=n; m\geq 1;k\geq 0}c_\psi (m,r)(\omega(k)+\omega(k-m)+\cdots ).
\end{equation}
\end{theorem}
\begin{proof}
Let $c(n,r)$ be the number of compositions of $n$ with exactly $r$ parts. Gould  \cite{gould} observed that,
\begin{equation}
c(n,r)=\sum_{d|n}c_\psi(d,r).
\end{equation}
Catalan \cite{catalan} counted that,
\begin{equation}\label{comp-2}
c(n,r)={{n-1}\choose{r-1}}.
\end{equation}
Now the result follows from Theorem \ref{etr}.
\end{proof}
\subsection{ Representations as sum of squares}
Let $r_k(n)$ be the number of ways integer $n$ can be represented as a sum of $k$ squares. The following result due to Jacobi \cite{jacobi2} puts $r_2(n)$, $r_4(n)$ and $r_8(n)$ separately as a divisor sum of simple functions. This representation paves the way for deriving Euler-type recurrence relations for $r_i(n)$ $(i=2,4,8)$ with the aid of Theorem \ref{etr}.  
\begin{lemma}[Jacobi]
We have
\begin{enumerate}
\item[(a)]
\[\frac{r_2(n)}{4}=\sum_{d|n}\eta_1(d),
\]
where 
\[\eta_1(n)=\begin{cases}
0&\text{ if }n\text{ is even}\\
(-1)^{\frac{n-1}{2}}&\text{ if }n\text{ is odd},
\end{cases}
\]
\item[(b)]
\[r_4(n)=8\sum_{d|n}\eta_2(d),
\]
where 
\[\eta_2(n)=\begin{cases}
0&\text{ if }4|n\\
n&\text{ otherwise},
\end{cases}
\]
\item[(c)]
\[r_8(n)=16\sum_{d|n}(-1)^{n+d}d^3.
\]
\end{enumerate}
\end{lemma}
Now, in accordance with Theorem \ref{etr}, we  have the following result. 
\begin{theorem}\label{two-squares}
We have
\begin{enumerate}
\item[(a)]
\begin{equation}
\sum_{k=1}^{n}r_2(k)\omega(n-k)={4}\left(\sum_{m+k=n; m\geq 1;k\geq 0}\eta_1(m)(\omega(k)+\omega(k-m)+\cdots )\right),
\end{equation}
\item[(b)]
\begin{equation}
\sum_{k=1}^{n}r_4(k)\omega(n-k)={8}\left(\sum_{m+k=n; m\geq 1;k\geq 0}\eta_2(m)(\omega(k)+\omega(k-m)+\cdots )\right),
\end{equation}
\item[(c)]
\begin{equation}
\sum_{k=1}^{n}\frac{r_8(k)}{(-1)^k}\omega(n-k)={16}\left(\sum_{m+k=n; m\geq 1;k\geq 0}(-1)^mm^3(\omega(k)+\omega(k-m)+\cdots )\right).
\end{equation}
\end{enumerate}
\end{theorem}
\subsection{ Subsets relatively prime to a positive integer}
Melvyn B. Nathanson \cite{nathan} denoted the number of nonempty subsets and the number of subsets of cardinality $r$ of $\left\{1,2,\ldots,n\right\}$ such that  greatest common divisor of each subset is relatively prime to $n$, respectively, by $\Phi(n)$ and $\Phi_r(n)$; and he showed that 
\begin{equation}\label{rps-1}
\sum_{d|n}\Phi(d)=2^n-1
\end{equation}
and 
\begin{equation}\label{rps-2}
\sum_{d|n}\Phi_r(d)={{n}\choose{r}}.
\end{equation}
Presenting these identities appropriately in  Theorem \ref{etr} lead to the following result.
\begin{theorem}
We have
\begin{enumerate}
\item[(a)]
\begin{equation}
\sum_{k=1}^{n}(2^k-1)\omega(n-k)=\sum_{m+k=n; m\geq 1;k\geq 0}\Phi(m)(\omega(k)+\omega(k-m)+\cdots ),
\end{equation}
\item[(b)]
\begin{equation}
\sum_{k=1}^{n}{{k}\choose{r}}\omega(n-k)=\sum_{m+k=n; m\geq 1;k\geq 0}\Phi_r(m)(\omega(k)+\omega(k-m)+\cdots ).
\end{equation}
\end{enumerate}
\end{theorem}
Now we replace the set $\left\{1,2,\ldots,n\right\}$ in the above definition with the set of positive divisors of $n$. Here too the pattern of the above identities retains.
\begin{definition}
Let $n$ be a positive integer. Denote by $\Phi^\tau(n)$ \emph{(}resp. $\Phi^\tau_r(n)$\emph{)}, the number of nonempty subsets \emph{(}resp. the number of subsets of cardinality $r$\emph{)} of the set of positive divisors of $n$ such that greatest common divisor of each subset is relatively prime to $n$.
\end{definition}
\begin{lemma}
We have
\begin{enumerate}
\item[(a)]
\begin{equation}
\sum_{d|n}\Phi^\tau(d)=2^{\tau(n)}-1,
\end{equation}
\item[(b)]
\begin{equation}
\sum_{d|n}\Phi_r^\tau(d)={{\tau(n)}\choose{r}}.
\end{equation}
\end{enumerate}
\end{lemma}
\begin{theorem}
We have
\begin{enumerate}
\item[(a)]
\begin{equation}
\sum_{k=1}^{n}(2^{\tau(k)}-1)\omega(n-k)=\sum_{m+k=n; m\geq 1;k\geq 0}\Phi^\tau(m)(\omega(k)+\omega(k-m)+\cdots ), 
\end{equation}
\item[(b)]
\begin{equation}
\sum_{k=1}^{n}{{\tau(k)}\choose{r}}\omega(n-k)=\sum_{m+k=n; m\geq 1;k\geq 0}\Phi_r^\tau(m)(\omega(k)+\omega(k-m)+\cdots ).
\end{equation}
\end{enumerate}
\end{theorem}
\begin{note}
\emph{Comparing Equation (\ref{rps-1}) and Equation (\ref{comp-1}), we get}
\begin{equation}
\Phi(n)=\begin{cases}
2c_\psi(n)-1&\text{ if }n=1;\\
2c_\psi(n)&\text{ if }n\geq 2.
\end{cases}
\end{equation}
\emph{Comparing Equation (\ref{rps-2}) and Equation (\ref{comp-2}) in the light of Pascal's identity, one can get that}
\begin{equation}
\Phi_r(n)=c_\psi(n,r)+c_\psi(n,r+1).
\end{equation}
\end{note}
\section{ Another way to Euler-type recurrence relations}
\subsection{ Proof of Theorem \ref{etnew1}}
Now we look at some unnoticed things in a recent paper by Yuriy Choliy, Louis W. Kolitsch and Andrew V. Sills \cite{yuriy} that lead to an Euler-type recurrence relation for $qq(n)$ and a similar type of recurrence relation for $q(n)$ involving square numbers in place of pentagonal numbers, which is crux of Theorem \ref{etnew1}.
\begin{lemma}[Euler's partition theorem \cite{euler}]
The following equality holds:
\[
\sum_{n=0}^{\infty}q(n)q^n=\prod_{m=1}^{\infty}(1+q^m)=\prod_{m=1}^{\infty}\frac{1}{1-q^{2m-1}}.
\]
\end{lemma}
Consider the following identity due to Jacobi \cite{jacobi1}:
\begin{equation}\label{jtp}
\prod_{n=1}^{\infty}(1-q^{2n})(1+q^{2n-1})(1+q^{2n-1})=1+2\sum_{m=1}^{\infty}\delta_s(m)q^m.
\end{equation}
Replacing $q$ with $-q$ gives
\begin{equation}\label{j1}
\prod_{n=1}^{\infty}(1-q^n)\prod_{n=1}^{\infty}(1-q^{2n-1})=1+2\sum_{m=1}^{\infty}(-1)^m \delta_s(m)q^m.
\end{equation}
When products in the left extreme were expanded then one can get (a) of Theorem \ref{etnew1}.

Equation (\ref{j1}) can be put as
\[\prod_{n=1}^{\infty}(1-q^n)=\frac{1+2\sum_{m=1}^{\infty}(-1)^m \delta_s(m)q^m}{\prod_{n=1}^{\infty}(1-q^{2n-1})}.
\]
Then by Euler's partition theorem, we have
\[\prod_{n=1}^{\infty}(1-q^n)=\left(1+2\sum_{m=1}^{\infty}(-1)^m \delta_s(m)q^m\right)\left(\sum_{m=0}^{\infty}q(m)q^m\right).
\]
Now (b) of Theorem \ref{etnew1} follows by Euler's Pentagonal Number Theorem.  
\subsection{ Proof of Theorem \ref{etnew2}}
Euler's partition theorem and Gauss identity \ref{gauss1} implies
\begin{equation}
\sum_{n=0}^{\infty}q^{\frac{n(n+1)}{2}}=\left(\prod_{m=1}^{\infty}(1-q^{2m})\right)\left(\sum_{m=0}^{\infty}q(m)q^m\right).
\end{equation}
Since
\[\prod_{m=1}^{\infty}(1-q^{2m})=1+\omega(1)q^2+\omega(2)q^4+\omega(3)q^6+\cdots ,
\]
(a) of Theorem \ref{etnew2} follows. 

Another form of Gauss identity \ref{gauss1}, namely,
\[\left(\prod_{n=1}^{\infty}{\frac{1}{1-q^{2n}}}\right)\left(\sum_{n=0}^{\infty}q^{\frac{n(n+1)}{2}}\right)=\sum_{m=0}^{\infty}q(m)q^m
\]
gives (b) of Theorem \ref{etnew2}, since
\[\prod_{n=1}^{\infty}\frac{1}{1-q^{2n}}=p(0)+p(1)q^2+p(2)q^4+p(3)q^6+\cdots .
\] 
Gauss identity \ref{gauss1} can also be put as
\[\left(\sum_{n=0}^{\infty}q^{\frac{n(n+1)}{2}}\right)\left(\prod_{n=1}^{\infty}(1-q^{2n-1})\right)=\prod_{m=1}^{\infty}(1-q^{2m}).
\]
Since
\[\prod_{m=1}^{\infty}(1-q^{2m-1})=\sum_{n=0}^{\infty}(-1)^nqq(n)q^n,
\]
(c) of Theorem \ref{etnew2} follows. 
\subsection{Proof of Theorem \ref{comp}}
Let $A$ be a set of positive integers. We define 
\[\chi_A(q)=\sum_{a\in A}q^a.
\]
Let $c_A(n)$ be the number of  compositions of $n$ with parts from set $A$. Heubach and Mansour \cite{heubach} documented that, 
\begin{equation}
\sum_{n=0}^{\infty}c_A(n)q^n=\frac{1}{1-\chi_A(q)}.
\end{equation}
Now, in accordance with the Gauss identity \ref{gauss2}, we can write
\begin{align*}
1+\sum_{n=1}^{\infty}(-1)^ns(n)q^n&=\frac{1}{1-\sum_{n=1}^{\infty}(-1)^{n^2}q^{n^2}}\\
&=\frac{1}{1-\left(\frac{\prod_{m=1}^{\infty}\frac{1-q^m}{1+q^m}-1}{2}\right)}\\
&=\frac{2}{3-\prod_{m=1}^{\infty}\frac{1-q^m}{1+q^m}}\\
&=\frac{2\prod_{m=1}^{\infty}(1+q^m)}{3\prod_{m=1}^{\infty}(1+q^m)-\prod_{m=1}^{\infty}(1-q^m)}.
\end{align*}
This gives
\[\left(1+\sum_{n=1}^{\infty}(-1)^ns(n)q^n\right)\left(3\prod_{m=1}^{\infty}(1+q^m)-\prod_{m=1}^{\infty}(1-q^m)\right)=2\prod_{m=1}^{\infty}(1+q^m).
\]
Equating the coefficients of like powers of $q$ in both extremes gives (a) of Theorem \ref{comp}. A similar application of Gauss identity \ref{gauss1} gives (b) of Theorem 4. 
\section{ In the way of Euler's logarithmic derivative}
In this section, we wield the logarithmic derivative technique of Euler \cite{euler2} to conclude Theorem \ref{ld1}.
\subsection{ Proof of Theorem \ref{ld1}}
Taking log on both sides of Equation (\ref{jtp}), we have
\[\sum_{n=1}^{\infty}\log{(1-q^{2n})}+2\sum_{n=1}^{\infty}\log{(1+q^{2n-1})}=\log{\left(1+2\sum_{m=1}^{\infty}\delta_s(m)q^m\right)}.
\]
Differentiating and then multiplying by $q$, we get
\[-\sum_{n=1}^{\infty}\frac{2nq^{2n}}{1-q^{2n}}+2\sum_{n=1}^{\infty}\frac{(2n-1)q^{2n-1}}{1+q^{2n-1}}=\frac{2\sum_{m=1}^{\infty}m\delta_s(m)q^{m}}{1+2\sum_{m=1}^{\infty}\delta_s(m)q^m}.
\]
Replacing $q$ with $-q$, we have then by the Lambert series representation that,
\[\left(1+2\sum_{n=1}^{\infty}(-1)^n\delta_s(n)q^n\right)\left(\sum_{n=1}^{\infty}(\sigma(n)+\sigma_o(n))q^n\right)=2\sum_{m=1}^{\infty}(-1)^{m+1}m\delta_s(m)q^m.
\]
Now (a) of Theorem \ref{ld1} follows. One can apply the same method  in Gauss identity \ref{gauss1} to obtain (b) of Theorem \ref{ld1}.
To realize (c) of Theorem \ref{ld1}, the following lemma is an essential one.
\begin{lemma}\label{sigmas}
We have
\begin{equation}
\sum_{n=1}^{\infty}n\frac{q^n}{1+q^n}=\sum_{n=1}^{\infty}\sigma_s(n)q^n.
\end{equation}
\end{lemma}
\begin{proof}
 Suppose that $n=dk$ for some positive integers $d$ and $k$. Then the coefficient of $q^n$ in the binomial expansion of $k\frac{q^k}{1+q^k}$ is $(-1)^{d-1}\frac{n}{d}$. Therefore the coefficient of $q^n$ in the sum $\sum_{k=1}^{\infty}k\frac{q^k}{1+q^k}$ is $\sum_{d|n}(-1)^{d-1}\frac{n}{d}$ as expected.
\end{proof}
Apply the logarithmic derivative technique in Gauss identity \ref{gauss2}, then (c) of Theorem \ref{ld1} follows in the light of Lemma \ref{sigmas}.
\subsection{ An extension of Theorem \ref{ld1}}
Again consider the Jacobi's identity \cite{jacobi1}
\[
\prod_{n=1}^{\infty}(1-q^{2n})(1+q^{2n-1})(1+q^{2n-1})=\sum_{m=-\infty}^{\infty}q^{m^2}.
\]
Then for every integer $k\geq 2$, we have
\[\left(\prod_{n=1}^{\infty}(1-q^{2n})(1+q^{2n-1})(1+q^{2n-1})\right)^k=1+\sum_{m=1}^{\infty}r_k(m)q^m,
\]
where $r_k(m)$, as denoted earlier, represents the number of ways in which $m$ can be written as a sum of $k$ squares.

The transformation $q\to -q$ gives
\[\left(\prod_{n=1}^{\infty}(1-q^n)\right)^k\left(\prod_{n=1}^{\infty}(1-q^{2n-1})\right)^k=1+\sum_{m=1}^{\infty}(-1)^mr_k(m)q^m.
\]
Applying the operator $q\frac{d}{dq}$ and using the Lambert series representation as before, we have the following recurrence relation for $r_k(n)$. 
\begin{theorem}
We have
\begin{equation}
k\left(\sum_{i=1}^{n}(\sigma(i)+\sigma_o(i))((-1)^{n-i}r_k(n-i))\right)=(-1)^{n+1}nr_k(n).
\end{equation}
\end{theorem}
As an immediate consequence, we get the following congruence property of $r_k(n)$.
\begin{corollary}
Let $n$ and $k$ be positive integers such that $\gcd(k,n)=1$. We have
\begin{equation}
r_k(n)\equiv 0(mod\ k).
\end{equation}
\end{corollary}

\end{document}